[1994/12/01]
\documentclass[12pt]{amsart}
\title{Finite time blow up for a 1D model of 2D Boussinesq system}

\author{Kyudong Choi}
\thanks{Department of
Mathematics, University of Wisconsin, Madison, WI 53706, USA; email: kchoi@math.wisc.edu}
\author{Alexander Kiselev}
\thanks{Department of
Mathematics, University of Wisconsin, Madison, WI 53706, USA; email: kiselev@math.wisc.edu}
\author{Yao Yao}
\thanks{Department of
Mathematics, University of Wisconsin, Madison, WI 53706, USA; email: yaoyao@math.wisc.edu}

\date{\today }

\usepackage[ulem=normalem]{changes}

\usepackage{amsmath,amsthm,amsfonts,amsopn,verbatim,graphicx}
\usepackage[margin=1.2in]{geometry}
\chardef\bslash=`\\ 

\newtheorem{thm}{Theorem}[section]

\newtheorem{lem}[thm]{Lemma}
\newtheorem{prop}[thm]{Proposition}

\newtheorem*{thm*}{Theorem}
\newtheorem*{lem*}{Lemma}
\newtheorem*{prop*}{Proposition}
\theoremstyle{definition}

\theoremstyle{remark}
\newtheorem{rem}{Remark}[section]

\usepackage{amscd,amsthm}
\usepackage{}

\newcommand{\eval}[2][\right]{\relax
  \ifx#1\right\relax \left.\fi#2#1\rvert}

\begin{document}

\begin{abstract}
The 2D conservative Boussinesq system describes inviscid, incompressible, buoyant fluid flow in gravity field.
The possibility of finite time blow up for solutions of this system is a classical problem of mathematical hydrodynamics.
We consider a 1D model of 2D Boussinesq system motivated by a particular finite time blow up scenario. We prove that finite
time blow up is possible for the solutions to the model system.
\end{abstract}

\maketitle

\section{Introduction}

The 2D Boussinesq system for vorticity of the fluid $\omega(x,t)$ and density (or temperature) $\rho(x,t)$ is given by
\begin{eqnarray}\label{vortbous}
\partial_t \omega +(u \cdot \nabla)\omega = \partial_{x_1} \rho;\,\,\, \,\,\,
\partial_t \rho +(u \cdot \nabla)\rho = 0; \\
\nonumber u = \nabla^\perp (-\Delta)^{-1}\omega,\,\,\,\omega(x,0)=\omega_0(x),\,\,\,\rho(x,0)=\rho_0(x).
\end{eqnarray}
The 2D Boussinesq system models motion of buoyant incompressible fluid that takes place in atmosphere, ocean, inside Earth or stars, and in every kitchen. Global regularity of solutions is known when classical dissipation is
present in at least one of the equations \cite{Chae}, \cite{HouLi}, or under a variety of more general conditions on dissipation (see e.g. \cite{CaoWu1} for more information).
The regularity vs finite time blow up question for the inviscid 2D Boussinesq system \eqref{vortbous}
is a well known open problem that has appeared, for example, on the ``eleven great problems
of mathematical hydrodynamics" list proposed by Yudovich \cite{Yud2000}. There is also an interesting connection between \eqref{vortbous} and axi-symmetric three dimensional Euler equation: the equations are closely
related and virtually identical away from the rotation axis (see e.g. \cite{MB}, page 186).

There has been much numerical work on trying to find possible singular scenario for solutions of axi-symmetric 3D Euler equation with swirl or 2D Boussinesq system. Often, situations where strong growth
of solutions has been observed were later determined to be regular by further numerical or analytic research. For numerical studies, see for example \cite{PumSig}, \cite{EShu}, or a review \cite{Gibbon}.
Analytical tools for ruling out blow up scenario include nonlinearity depletion mechanisms discovered by Constantin, Fefferman and Majda \cite{CF1}, \cite{CFM} and later extensions in \cite{HL1}, \cite{HL2}.

In a recent work \cite{HouLuo1}, Tom Hou and Guo Luo suggested a new scenario for possible singularity formation in 3D Euler equation. In their scenario, the flow is axi-symmetric and confined in a rotating cylinder
with no flow condition on the boundary. The numerically observed growth of vorticity happens at the boundary of the cylinder, away from rotation axis. So one can equivalently work with \eqref{vortbous} set on
a square $D$ (corresponding to a fixed angular variable in the 3D case). 
Motivated by \cite{HouLuo1}, Kiselev and Sverak \cite{KS} considered a similar setting for the 2D Euler equation on a disk. The solutions of 2D Euler equation with smooth initial data
are well known to be globally regular. However the work  \cite{KS} constructs examples with double exponential growth of the vorticity gradient. This is known to be the fastest possible rate of growth, 
and \cite{KS} provides the first
example where such growth happens.
The growth in \cite{KS} also happens on the boundary, and their result confirms that the scenario of \cite{HouLuo1} is indeed an interesting candidate to consider for blow up in solutions of 3D Euler equation
or 2D Boussinesq system.

Compared to the 2D Euler case, the 2D Boussinesq system presents significant new difficulties for analysis. There are nonlinear effects coming from the coupling in \eqref{vortbous}, and possible growth in vorticity
makes solutions harder to control. A simplified one-dimensional model has been suggested in \cite{HouLuo1} and analyzed in \cite{HouLuo2}. It is given by
\begin{eqnarray}\label{HL1}
\partial_t \omega + u \partial_x \omega = \partial_x \rho; \,\,\, \partial_t \rho + u \partial_x \rho =0; \\
\nonumber u_x = H \omega,\,\,\,\omega(x,0)=\omega_0(x),\,\,\,\rho(x,0)=\rho_0(x)
\end{eqnarray}
where the the initial data is periodic with period two, the density function is even, the vorticity is odd with respect to $x=0$ and $x=1,$ 
and $H \omega$ denotes the periodic Hilbert transform of vorticity.
Local well-posedness and a number of useful estimates have been proved in \cite{HouLuo2} for the system \eqref{HL1}, and both numerical simulations as well as
formal arguments suggesting blow up have been carried out.
However a fully rigorous proof of finite time blow up is currently not available for the system \eqref{HL1}.

Our goal in this paper is to analyze a related but further simplified system that is inspired by \cite{KS}. The system is set on an interval $[0,1]$ with Dirichlet boundary
conditions for $\omega$ and $\rho.$
\begin{equation}
 \begin{cases}
&  \partial_t \rho(t,x) + u(t,x)\partial_x\rho(t,x)=0,\\[0.1cm]
&\partial_t \omega(t,x)+u(t,x)\partial_x \omega(t,x)=\partial_x\rho(t,x),\\[0.05cm]
&u(t,x)=-x\Omega(t,x),\quad\Omega(t,x)=\int_x^1\dfrac{\omega(t,y)}{y}dy,\\
&\omega(0,x)=\omega_0(x),\quad \rho(0,x)=\rho_0(x), \quad \omega_0(0)=\omega_0(1)=\rho_0(0)=\rho_0(1)=0.\\
 \end{cases}
 \label{eq:system}
\end{equation}

We choose to work with Dirichlet boundary conditions for both $\omega$ and $\rho,$ which is more natural than periodic setting for our version of the Biot-Savart law.
The Biot-Savart law linking fluid velocity to vorticity is the main difference between \eqref{HL1} and \eqref{eq:system}. 
The law for the system \eqref{eq:system} is simpler, even though closely related to the law for the system \eqref{HL1}. This facilitates
the analysis. 
Such simplified Biot-Savart law is motivated by the result proved in \cite{KS}. It is shown there that under certain conditions on the initial data $\omega_0,$ the flow near the origin $O$ is hyperbolic for all times.
Namely, apart from small exceptional sectors, the velocity $u$ near $O$ satisfies
\begin{eqnarray}\label{u1} u_1(x_1,x_2,t) = - \frac{4}{\pi} x_1\int_{Q(x_1,x_2)} \frac{y_1y_2}{|y|^4} \omega(y,t)\,dy_1dy_2 + x_1 B_1(x_1,x_2,t) \\
\label{u2} u_2(x_1,x_2,t) = \frac{4}{\pi}x_2\int_{Q(x_1,x_2)} \frac{y_1y_2}{|y|^4} \omega(y,t)\,dy_1dy_2 + x_2 B_2(x_1,x_2,t),
\end{eqnarray}
where $x_1,x_2 \geq 0,$ $|B_{1,2}(x_1,x_2,t)| \leq C(\gamma)\|\omega\|_{L^\infty}$ and $Q(x_1,x_2) = \{ y | y \in D, \,\,\,x_1 \leq y_1,\,\,\, x_2 \leq y_2 \}.$
The first term on the right hand side of \eqref{u1}, \eqref{u2} is the main term, and it is this term that is modeled by $u(x,t) = -x \int_x^1 \omega(y,t)/y \,dy$ in \eqref{eq:system}.
Thus one can expect the system \eqref{eq:system} to be a reasonable model of the true 2D Boussinesq dynamics as far as the hyperbolic flow formulas like \eqref{u1}, \eqref{u2}
remain valid, in particular all the time up to blow up if it happens. This is far from clear, even though the numerical simulations of Hou and Luo \cite{HouLuo1} seem to suggest that this might
be the case.

In the first two sections below we will establish local well-posedness and conditional regularity results for the system \eqref{eq:system}, in particular proving an analog of the
celebrated Beale-Kato-Majda criterion \cite{BKM}.
Then we will prove our main result
\begin{thm}\label{mainthm}
There exist $\omega_0,\rho_0 \in C_0^\infty([0,1])$ for which the solution of \eqref{eq:system} blows up in finite time. In particular,
\[ \int_0^{T^*} \|\omega(t)\|_{L^\infty}\,dt \rightarrow \infty \] for some $T^*<\infty.$
\end{thm}
Roughly speaking, the blow-up proof is done by tracking the evolution of $\Omega(x,t)$ along a family of characteristics originating from a sequence of points $x_1\geq x_2 \geq  \dots,$ where $x_\infty := \lim_{n\to\infty} x_n>0$ satisfies $\rho_0(x_\infty)>0$. By obtaining lower bound on $\Omega$ on this family of characteristics, we conclude that the characteristic originating from $x_\infty$ must touch the origin before some finite time $T$, which implies that the classical solution has to break down at (or before) time $T$.

The main new effect reflected in Theorem~\ref{mainthm} is a rigorous understanding of the mechanism how coupling in 2D Boussinesq can in principle lead to blow up.
The main simplifications the result utilizes 
are lack of two-dimensional geometry which makes certain monotonicity properties easier to control as well as 
 reliance on
the stable hyperbolic form of fluid velocity akin to \cite{KS}. These simplifications are clearly significant, but one has to take the first step.

\section{Local well-posedness}

It will be often useful for us to solve equations for $\omega$ and $\rho$ on characteristics.
Denote $\phi_t(x)$ the solution of $$\begin{cases}
                                          &\dfrac{d}{dt}\phi_t(x)=u(t,\phi_t(x)),\\
                                          &\phi_0(x)=x.\\
                                         \end{cases}$$
                                          Then we have
                                         $$\rho(t,\phi_t(x))=\rho_0(x),$$
                                         $$\omega(t,\phi_t(x))=
                                         \omega_0(x)+
                                         \int_0^t(\partial_x\rho)
                                         (s,\phi_s(x))ds.$$

First we consider the following lemma which says that $u$ has almost one more derivative
than $\omega$ has.
\begin{lem} Let $\omega\in C_0^\infty((0,1))$ be a smooth function
that is compactly supported in $(0,1)$.\\
 Then we have
 $\begin{cases}
  &\|u\|_{H^{m+1}}\leq C_m\cdot\|\omega\|_{H^{m}} \mbox{ for } m\geq 0
 \mbox{ and}\\
  &\|u^{(m+1)}\|_{L^\infty}\leq C_m\cdot\|\omega^{(m)}\|_{L^\infty} \mbox{ for } m\geq 1.
 \end{cases}$
\end{lem}
\begin{proof}
Observe that $\|u\|_{L^\infty}\leq \|\omega\|_{L^1}$ and
 $u^\prime=-\Omega+\omega$.
 For $p\in[1,\infty)$, we obtain $\|\Omega\|_{L^p}\leq p\cdot\|\omega\|_{L^p}$ by using the
 following Hardy's inequality with $f(x)=\omega(x)/x:$
 $$\Big(\int_0^\infty\Big(\int_x^\infty|f(x)|dx\Big)^pdx\Big)^{1/p}\leq
 p\Big(\int_0^\infty|f(x)|^px^pdx\Big)^{1/p}.
 $$ It shows that $\|u\|_{H^1}\leq \|\omega\|_{L^2}$.

 For $u\in H^{m+1}$ estimate with $m\geq1$, observe that $u^{(m+1)}(x)=\sum_{i=0}^m
 C_{m,i}\cdot\frac{\omega^{(m-i)}(x)}{x^i}$ for some constants $C_{m,i}$. We claim
 $\|\frac{\omega^{(m-i)}(x)}{x^i}\|_{L^2}\leq C \|\omega^{(m)}\|_{L^2}$.
 Indeed, observe that for $n\geq1$ and for smooth $f$ which is compactly
 supported in $(0,1)$,
 \begin{equation*}\begin{split}
 \int_0^1\Big(\frac{f(x)}{x^{n}}\Big)^2dx&=\frac{f^2(x)}{(1-2n)x^{2n-1}}\Big|_{x=0}^{x=1}+\int_0^1\frac{2ff^\prime}{(2n-1)x^{2n-1}}dx \\
 &\leq
 \frac{2}{2n-1}\Big( \int_0^1\Big(\frac{f(x)}{x^{n}}\Big)^2dx\Big)^{1/2}
 \Big( \int_0^1\Big(\frac{f^\prime(x)}{x^{n-1}}\Big)^2dx\Big)^{1/2}.
 \end{split}\end{equation*} 

 This gives us $ \|f(x)/x^n\|_{L^2}\leq C \|f^\prime(x)/x^{n-1}\|_{L^2}$. We can iterate
 until we  get  $\|u\|_{H^{m+1}}\leq C\|\omega\|_{H^m}$.

 The $u^{(m+1)}\in L^\infty$ estimate follows from Taylor error estimates $|f(x)/x^n|\leq
 C\|f^{(n)}\|_{L^\infty}$.

\end{proof}
\begin{rem}
The above $u^{(m+1)}\in L^\infty$ estimate does not hold for the case $m=0$.
Instead, we have only  pointwise estimate:
\begin{equation*}   |u^\prime(x)|\leq \|\omega\|_{L^\infty}\cdot(1
   -\ln(x)) \mbox{ for }
   x\in(0,1).\end{equation*} However, it will be proved for a solution $\omega$ on $[0,T)$ with finite $T$ that
     $\int_0^{T}\|\omega(t)\|_{L^\infty}dt<\infty$ implies
   $\int_0^{T}\|\partial_xu(t)\|_{L^\infty}dt<\infty$ (see Proposition \ref{bkm}).

\end{rem}
\begin{rem}
 We can weaken the condition that $\omega$ is compactly supported in $(0,1)$. For example,
in order
 to get  $u\in H^{m+1}$ estimate
  assuming $\omega\in H^m_0((0,1))$ is enough  
 {(where
 $H^m_0((0,1))$ is the completion of $(C^\infty_0\cap H^m)((0,1))$ by using the topology of
 $H^m((0,1)))$.}
 Recall
that
we used the fact that $\omega$ is compactly supported in $(0,1)$ only to say the boundary term
  $\Big(\dfrac{f(x)}{x^{n-(1/2)}}\Big)^2\Big|_{x=0}^1$  from integration by parts  vanishes.
  From Sobolev embedding,
 $\omega\in H^m_0$ implies $\omega\in C^{m-1}$ and
 $\omega^{(i)}(0)=\omega^{(i)}(1)=0$ for $i=0,1,\dots,(m-1)$. Moreover, the embedding
gives us $\omega^{(m-1)}\in {C^{1/2}}$-Holder space, which implies
$\dfrac{\omega^{(m-1)}(x)}{\sqrt{x}}\leq C \|\omega\|_{H^m}$.
Taking $\omega\in H^m_0$ suffices to carry out the same computation  in the same manner as for
compactly supported function. 
Similarly, it is enough for $u^{(m+1)}\in L^\infty$ estimate
to assume $\omega^{(m)}\in L^\infty$ and $\omega^{(i)}(0)=\omega^{(i)}(1)=0$ for $i=0,1,\dots,(m-1)$ instead of assuming that $\omega$ is compactly supported in $(0,1)$.
\end{rem}

\begin{prop}\label{exist}
 Given any
 initial data
 { $(\omega_0,\rho_0)\in H^m_0((0,1))\times H^{m+1}_0((0,1))$}
 with $m\geq2,$ 
 there exists
 $T=T(\|\omega_0\|_{H^m}+\|\rho_0\|_{H^{m+1}})>0$ such that the system has a unique
 classical solution $(\omega,\rho)\in C([0,T];H_0^{m}\times H_0^{m+1})$.
\end{prop}

\begin{proof}

Consider a function $\psi\in C^\infty(\mathbb{R})$ such that $\int \psi =1$,
$\psi\geq 0$ and $supp(\psi)\subset[-1,1],$ and set $\psi_\epsilon(x):=
\psi(x/\epsilon)/\epsilon$ for $\epsilon>0$. 
First we replace the initial data $(\omega_0,\rho_0)$ with approximations compactly supported in $(0,1)$,
given by $(\tilde{\omega_0},\tilde{\rho_0})(x):=(\omega_0,\rho_0)(\frac{x-2\epsilon}{1-4\epsilon})$.
Then we mollify the initial data $(\tilde{\omega_0},\tilde{\rho_0})$ by convolution:
$\omega^\epsilon_0:=\tilde{\omega_0}*\psi_\epsilon$ and $\rho^\epsilon_0:=
\tilde{\rho_0}*\psi_\epsilon$. Note that $\omega_0^\epsilon$ and $\rho_0^\epsilon$
lie in $C^\infty$ and they are
compactly supported in $[\epsilon,1-\epsilon]\subset(0,1)$.\\
 Define $u_0^\epsilon(t,x):=-x\int_x^{1}\frac{\omega^\epsilon_{0}(y)}{y}dy$. Then
 consider the following iteration scheme for $n\geq1:$
 \begin{equation}\label{iteration}
\begin{cases}
&  \partial_t \rho_n^\epsilon + u_{n-1}^\epsilon\partial_x\rho_n^\epsilon=0 \mbox{ with }
\rho^\epsilon(0)
=\rho_0^\epsilon, \\
&\partial_t \omega_n^\epsilon+u_{n-1}^\epsilon\partial_x \omega_n^\epsilon=\partial_x\rho_n^\epsilon
\mbox{ with } \omega_n^\epsilon(0)
=\omega_0^\epsilon,\\
&u_{n}^\epsilon(t,x)=-x
\int_x^{1}\frac{\omega_{n}^\epsilon(t,y)}{y}dy.
 \end{cases}
\end{equation}
Namely, for each $n\geq1$, we can solve the characteristic equations
$$\begin{cases}
  & \frac{d}{dt}\phi_n^\epsilon(t,x)=u_{n-1}^\epsilon(t,\phi_n(t,x)),\\
  & \phi_n^\epsilon(0,x)=x
 \end{cases}$$ for $t\in[0,\infty)$ since $u_{n-1}^\epsilon\in C^\infty_{t,x}$. Then define $\rho_n^\epsilon, \omega_n^\epsilon$ for $t\in[0,\infty)$ via the characteristics so that
                                         $\rho_n^\epsilon(t,\phi_n^\epsilon(t,x))=
                                         \rho_0^\epsilon(x)$ and
                                         $\omega_n^\epsilon(t,\phi_n^\epsilon(t,x))=
                                         \omega_0^\epsilon(x)+
                                         \int_0^t(\partial_x\rho_n^\epsilon)
                                         (s,\phi_n^\epsilon(s,x))ds$. Note that
                                         this process can be repeated and
                                         we get
                                         $\rho_n^\epsilon,\omega_n^\epsilon\in
                                         C^\infty_{t,x}$ 
                                         which are 
                                         are compactly
                                         supported in
                                         $(0,1)$
                                         for each $t>0$ since
                                         $x=0$ and $1$ are stationary points under the flow.\\

Let $m\geq2$. Simple energy estimates give us that, 
for any $n\geq1$,
$$
\frac{d}{dt}\Big(
\|\omega_n^\epsilon(t)\|^2_{H^m}
+\|\rho_n^\epsilon(t)\|^2_{H^{m+1}}
\Big)\leq
C\Big(\|u_{n-1}^\epsilon(t)\|_{H^{m+1}}+1\Big)\Big(
\|\omega_n^\epsilon(t)\|^2_{H^m}
+\|\rho_n^\epsilon(t)
\|^2_{H^{m+1}}
\Big).
$$ 
Since  $\rho_n^\epsilon(t),\omega_n^\epsilon(t)$ are compactly
                                         supported in $(0,1)$, we have
 $\|u_{n-1}^\epsilon(t)\|_{H^{m+1}}\leq C
\|\omega_{n-1}^\epsilon(t)\|_{H^{m}}$ by the previous lemma. As a result,
we obtain $\begin{cases}&\frac{d}{dt}f_n^\epsilon(t)\leq C\sqrt{
f_{n-1}^\epsilon(t)}f_n^\epsilon(t),\\& f_n^\epsilon(0)=f^\epsilon_0\end{cases}$ where
$f_n^\epsilon(t):=
\|\omega_n^\epsilon(t)\|^2_{H^m}
+\|\rho_n^\epsilon(t)\|^2_{H^{m+1}}
+1$ and
$f^\epsilon_0:=
\|\omega_0^\epsilon\|^2_{H^m}
+\|\rho_0^\epsilon\|^2_{H^{m+1}}
+1$. After a straightforward monotonicity
argument, this implies
\begin{equation}\label{unif_esti_n_e}
 (f_n^\epsilon(t))\leq 1/((f^\epsilon_0)^{-1/2}-Ct)^2,\quad\mbox{for }n\geq 1
 \mbox{ and for } 0\leq t< C/\sqrt{f^\epsilon_0}.
\end{equation}

Denote $f_0:=
\|\omega_0\|^2_{H^m}
+\|\rho_0\|^2_{H^{m+1}}
+1$.
Take $T$ between
$0$ and $C/\sqrt{f_0}$.
Thanks to the fact that
$f^\epsilon_0$ converges to $f_0$ as $\epsilon\rightarrow 0$,
 we know
$T<C/\sqrt{f^\epsilon_0}$ for sufficiently small $\epsilon>0$.
Then, for small $\epsilon>0$, we get
\begin{equation}\label{bounded}\sup_{t\in[0,T]}\Big(
\|\omega_n^\epsilon(t)\|_{H^m}+\|\rho_n^\epsilon(t)\|_{H^{m+1}}
\Big)<\infty\end{equation}
and, by using the structure of \eqref{iteration},
\begin{equation}\label{time_derivative_bounded}\sup_{t\in[0,T]}\Big(
\|\partial_t\omega_n^\epsilon(t)\|_{H^{m-1}}+\|\partial_t\rho_n^\epsilon(t)\|_{H^m}
\Big)<\infty.\end{equation}

Note that the above estimates are uniform in $n\geq 1$. Then the existence of a solution
 $(\omega^\epsilon,\rho^\epsilon)\in C([0,T];H_0^{m}\times H_0^{m+1})$
to \eqref{eq:system} corresponding the mollified initial data
$(\omega^\epsilon_0,\rho^\epsilon_0)$ follows the standard argument (e.g. see \cite{MB}).\\

We briefly sketch this argument here. First there exists a weak-$*$ limit
$(\omega^\epsilon,\rho^\epsilon)\in L^\infty(0,T;H_0^{m}\times H_0^{m+1}),$ which follows from
\eqref{bounded}
by Banach-Alaoglu theorem. Then, by using
\eqref{bounded} and \eqref{time_derivative_bounded},
we can show strong convergence
$(\omega_n^\epsilon,\rho_n^\epsilon)\rightarrow (\omega^\epsilon,\rho^\epsilon)$ in
$C([0,T];H^{m-\delta}
\times H^{m+1-\delta})$ for all real $\delta>0$.  Recall that we assumed $m\geq2$. Thus, from Sobolev's inequality,
all terms
in \eqref{eq:system} become continuous (pointwise). Moreover \eqref{iteration} converges
pointwise to \eqref{eq:system}.
It shows that  $(\omega^\epsilon,\rho^\epsilon)$ is a classical solution to \eqref{eq:system}.
Since $H^{-(m-\delta)}\times
H^{-(m+1-\delta)}$ is dense in $H^{-m}\times H^{-(m+1)}$, our solution
$(\omega^\epsilon,\rho^\epsilon)$ is weakly continuous in time variable as
a $H_0^{m}\times H_0^{m+1}$ valued function. Lastly, thanks to
weak continuity in time and the estimate \eqref{unif_esti_n_e},
we can show
$(\omega^\epsilon,\rho^\epsilon)\in
C([0,T];H_0^{m}
\times H_0^{m+1})$ by showing that
both $\|\omega_n^\epsilon(t)\|_{H^m}$ and $\|\rho_n^\epsilon(t)\|_{H^{m+1}}$ are
continuous in time variable $t\in[0,T]$. In addition, we have
\begin{equation}\label{unif_esti_e}
 f^\epsilon(t)\leq 1/((f^\epsilon_0)^{-1/2}-Ct)^2,\quad
 \mbox{ for } 0\leq t\leq T.
\end{equation}

To find a solution for the original initial data $(\omega_0,\rho_0)$,
recall that $T$ does not depend on $\epsilon$, the estimate
\eqref{unif_esti_e} is uniform in $\epsilon>0$,
 and $f^\epsilon_0$ converges to
$f_0$ as $\epsilon\rightarrow 0$. Then we repeat the above procedure as $\epsilon\rightarrow 0$ in order
to get a solution $(\omega,\rho)\in C([0,T];H_0^{m}\times H_0^{m+1})$
 to \eqref{eq:system} corresponding to $(\omega_0,\rho_0)$ with the same estimate
 \begin{equation*}
\Big(
\|\omega(t)\|^2_{H^m}
+\|\rho(t)\|^2_{H^{m+1}}
+1\Big)
 \leq 1/((f_0)^{-1/2}-Ct)^2\quad
 \mbox{ for } 0\leq t\leq T.
\end{equation*}
  Its uniqueness in the space $C([0,T];H_0^{m}\times H_0^{m+1})$ is easy to show (e.g. see \cite{ChaeNam}).




\end{proof}

\section{Beale-Kato-Majda type criteria}

\begin{prop}\label{bkm}
 Let 
 {$(\omega,\rho)\in C([0,T);H^{m}_0\times H_0^{m+1})$}  be the unique solution provided by Proposition~\ref{exist}
 for initial data $(\omega_0,\rho_0)\in H^{m}_0\times H^{m+1}_0$ with $m\geq2$.
 Then for any finite $T^*\leq T$, the followings are equivalent:\\
 \noindent (1). $\sup_{t\in[0,T^*]}\Big(
 \|\omega(t)\|_{H^m}+ \|\rho(t)\|_{H^{m+1}}
  \Big)<\infty$.\\
 (2). $\int_0^{T^*}\|\partial_xu(t)\|_{L^\infty}dt<\infty$.\\
 (3). $\int_0^{T^*}\|\omega(t)\|_{L^\infty}dt<\infty$.\\
 (4). $\int_0^{T^*}\|\partial_x\rho(t)\|_{L^\infty}dt<\infty$.\\
\end{prop}
\begin{rem}
  It is well known that for a full 2D inviscid Boussinesq  system, either
$\int_0^{T^*}\|\nabla u(t)\|_{L^\infty}dt<\infty$ or
$\int_0^{T^*}\|\nabla \rho(t)\|_{L^\infty}dt<\infty$
  implies (1) (see e.g. \cite{ChaeNam}, \cite{CaoWu1}).
  Whether (3) implies (1) for 2D inviscid Boussinesq system is an interesting open question.
\end{rem}

\begin{proof}
The implication $(1)\Rightarrow (2),(3)$ and $(4)$ is obvious 
from Sobolev's inequality.

The direction $(2)\Rightarrow (1)$ follows from a standard energy estimate. Indeed,
 if we denote $M:=\int_0^{T^*}\|\partial_xu(t)\|_{L^\infty}dt<\infty$, then
 we get
  for any $t\in[0,T^*]$,
 \begin{equation*}\begin{split}
                  &\|\partial_x\rho(t)\|^2_{L^2}\leq e^{CM} \|\partial_x\rho_0\|^2_{L^2},\\
         &\|\omega(t)\|^2_{L^2}\leq e^{CM}(1+T^*) (\|\omega_0\|^2_{L^2}+\|\partial_x\rho_0\|^2_{L^2}),\\
     &\|\partial_x\rho(t)\|_{L^\infty}\leq e^{M} \|\partial_x\rho_0\|_{L^\infty}, \mbox{ and}\\
          &\|\omega(t)\|_{L^\infty}\leq \|\omega_0\|_{L^\infty} +e^{M}T^* \|\partial_x\rho_0\|_{L^\infty}.\\
     \end{split}\end{equation*}
     Then straightforward estimates lead us to
      \begin{equation*}\begin{split}
               & \|\omega^\prime(t)\|_{L^2}+\|\rho^{\prime\prime}(t)\|_{L^2}\leq
               C_{M,T^*,
\|\omega_0\|_{H^1},  \|\rho_0\|_{H^{2}}
               }\quad\mbox{and}\\
                 & \|\omega^\prime(t)\|_{L^\infty}+\|\rho^{\prime\prime}(t)\|_{L^\infty}
                 \leq
                 C_{M,T^*,
\|\omega_0\|_{W^{1,\infty}}, \|\rho_0\|_{W^{2,\infty}}.
                 }\\
     \end{split}\end{equation*} where $W^{n,p}$ is the usual Sobolev space.
     We repeat this procedure until we get
    \begin{equation*} \|\omega^{(m)}(t)\|_{L^2}+\|\rho^{(m+1)}(t)\|_{L^2}\leq C_{M,T^*,
 \|\omega_0\|_{H^m}, \|\rho_0\|_{H^{m+1}}
    }.\end{equation*}

    For (4)$\Rightarrow$(3), we use the characteristic representation
    for $\omega$: $$\omega(t,\phi_t(x))=
                                         \omega_0(x)+
                                         \int_0^t(\partial_x\rho)
                                         (s,\phi_s(x))ds.$$

    For the direction  $(3)\Rightarrow(2)$, we denote 
 $M:=\int_0^{T^*}\|\omega(t)\|_{L^\infty}dt<\infty$. Then    
    we make an $L^\infty$-estimate for $\partial_x u$ in the
    following way.\\

    1. From $|u(t,x)|\leq \|\omega(t)\|_{L^\infty}\cdot x \cdot (-\ln(x))$, we
    get $\phi_t(x)\geq x^{\exp(\int_0^{t}\|\omega(s)\|_{L^\infty}ds)}\geq x^{\exp{(M)}}$ for
  $t\leq T^*$.
  We also get $\phi_{-t}(x)\leq x^{\exp(-M)}$.

  2. From $\partial_x u= -\Omega +\omega$, we get
  \begin{equation*}\begin{split}
     |(\partial_x u)(t,\phi_t(x))|&\leq 
     |\omega(t,\phi_t(x))|+|\Omega(t,\phi_t(x))|
  \leq \|\omega(t)\|_{L^\infty}(1+(-\ln(\phi_t(x)) )\\
  &\leq  \|\omega(t)\|_{L^\infty}(1+e^M(-\ln(x) )).
  \end{split}\end{equation*}

  3. From $ \partial_t (\partial_x\rho)  + u\partial_x(\partial_x\rho)
  =-(\partial_xu)(\partial_x\rho)$, we obtain
   \begin{equation*}\begin{split}
     |(\partial_x \rho)(t,\phi_t(x))|&\leq     |(\partial_x \rho_0)(x)|
     +\int_0^t|(\partial_x u)(s,\phi_s(x))|\cdot|(\partial_x \rho)(s,\phi_s(x))|ds.
  \end{split}\end{equation*} This implies
    \begin{equation*}\begin{split}
     |(\partial_x \rho)(t,\phi_t(x))|&\leq     |(\partial_x \rho_0)(x)|
      \exp\Big(\int_0^t|(\partial_x u)(s,\phi_s(x))|ds\Big)\\
     &\leq     |(\partial_x \rho_0)(x)|
      \exp\Big(\int_0^t\|\omega(s)\|_{L^\infty}(1+e^M(-\ln(x) )ds\Big)\\
       &\leq     |(\partial_x \rho_0)(x)|
      e^M \left(\frac{1}{x} \right)^{e^M\cdot M}.
  \end{split}\end{equation*}
  4. 
For a moment, assume that $M$ is so small that $M\cdot e^M\leq \frac{1}{2}$. Thanks to $\omega_0(0)=\partial_x\rho_0(0)=0$, we can estimate
  \begin{equation*}\begin{split}
     |\omega(t,\phi_t(x))|&\leq     |\omega_0(x)|
     +\int_0^t|(\partial_x \rho)(s,\phi_s(x))|ds\\
     &\leq  |\omega_0(x)| + |(\partial_x \rho_0)(x)|
     \cdot e^M\cdot \left(\frac{1}{x}\right)^{e^M\cdot M}\cdot T^*\\
      &\leq  \|\omega_0^\prime\|_{L^\infty}\cdot x + \|\rho_0^{\prime\prime}\|_{L^\infty}
      \cdot x
     \cdot e^M\cdot \left(\frac{1}{x}\right)^{1/2}\cdot T^*\\
      &
     \leq C_0 \sqrt{x} e^M (T^*+1)
  \end{split}\end{equation*} where $C_0:=\|\omega_0^\prime\|_{L^\infty}+
  \|\rho_0^{\prime\prime}\|_{L^\infty}$.
 So  we get a decay estimate of $\omega(t,x)$ near $x=0$:  \begin{equation*}\begin{split}
     |\omega(t,x)|&\leq    C_0 \sqrt{\phi_{-t}(x)} e^M (T^*+1)\leq
     C_0 x^{\frac{1}{2}\exp(-M)} e^M (T^*+1).
  \end{split}\end{equation*} This implies $L^\infty$ estimate of $\Omega$:
 \begin{equation*}\begin{split}
|\Omega(t,x)|&\leq \int_0^1\frac{|\omega(t,y)|}{y}dy\leq
C_0  e^M (T^*+1) \int_0^1 y^{\frac{1}{2}\exp(-M)-1}dy\leq
2C_0  e^{2M} (T^*+1).
\end{split}\end{equation*} Then we use $\partial_x u= -\Omega +\omega$
to get $$\|\partial_xu(t)\|_{L^\infty}\leq \|\omega(t)\|_{L^\infty}+2C_0e^{2M}(T^*+1)
 \quad\mbox{ for } t\in[0,T^*].$$

5. For general large $M$, we find $\sigma\in(0,T^*)$ such that
$M_{\sigma}:=\int_{\sigma}^{T^*} |\omega(s)|_{L^\infty}ds$ is so small that
$M_{\sigma}\cdot e^{M_{\sigma}}\leq \frac{1}{2}$. We do the same process not from $t=0$ but
from $t=\sigma$ to get
$$\|\partial_xu(t)\|_{L^\infty}\leq \|\omega(t)\|_{L^\infty}+2C_{\sigma}e^{2M}((T^*-\sigma)+1)
 \quad\mbox{ for } t\in[\sigma,T^*]$$
  where $C_{\sigma}:=
 \sup_{t\in[0,\sigma]}\Big(
  \|\omega(t)\|_{H_0^2}+
  \|\rho(t)\|_{H^3_0}\Big)$. Note that $C_\sigma$ is finite because $(\omega,\rho)$
  lies in $C([0,T);H_0^2\times H^3_0)$ and $\sigma<T^*\leq T$.\\

  Since
  $\|\partial_xu(t)\|_{L^\infty}\leq C\|u(t)\|_{H^2}\leq C\|\omega(t)\|_{H^1}\leq CC_\sigma$ for any $t\in[0,\sigma]$,
  we conclude
  $$\|\partial_xu(t)\|_{L^\infty}\leq \|\omega(t)\|_{L^\infty}+2C_{\sigma}e^{2M}(T^*+1)
  +CC_\sigma
 \quad\mbox{ for } t\in[0,T^*].$$
  \end{proof}

\section{Finite-time blow up examples}
Before we construct a finite-time blow up example, let us first state a lemma concerning the growth of $\Omega$ along the characteristics $\phi_t(x)$.
\begin{lem}
\label{lemma:Omega_t}
Along the characteristic $\phi_t(x)$, we have
  \begin{equation}
  \frac{d}{dt}\Omega(t,\phi_t(x))=\int_{\phi_t(x)}^1\frac{\omega(t,y)^2}{y}dy
  +\int_{\phi_t(x)}^1
 \frac{\partial_x\rho(t,y)}{y}dy.
 \label{eq:dt_Omega}
 \end{equation}
\end{lem}
\begin{proof}
Note that
\begin{equation}\begin{split}
\frac{d}{dt}\Omega(t,\phi_t(x))&=\partial_t\Omega(t,\phi_t(x))
+u(t,\phi_t(x))~\partial_x\Omega(t,\phi_t(x)).
\label{eq:dt_Omega_1}
  \end{split}\end{equation}
Let us compute $\partial_x \Omega$ and $\partial_t \Omega$ respectively. The definition of $\Omega$ directly gives that
  \begin{equation}
  \partial_x\Omega(t,x)=-\dfrac{\omega(t,x)}{x},
  \label{eq:dx_Omega}
  \end{equation}
  whereas $\partial_t \Omega(t,x)$ can be computed as follows:
\begin{equation}\begin{split}
 \partial_t\Omega(t,x)&=
 \int_x^1\frac{\partial_t \omega(t,y)}{y}dy=
  -\int_x^1\frac{u(t,y)\partial_x \omega(t,y)}{y}dy+
  \int_x^1\frac{\partial_x\rho(t,y)}{y}dy\\
  &=\int_x^1{\Omega(t,y)\partial_x \omega(t,y)}dy+
  \int_x^1\frac{\partial_x\rho(t,y)}{y}dy\\
 &=-\Omega(t,x) \omega(t,x)+\int_x^1\frac{\omega(t,y)^2}{y}dy+\int_x^1
 \frac{\partial_x\rho(t,y)}{y}dy.\end{split}
 \label{eq:dt_Omega_2}\end{equation}
In order to obtain \eqref{eq:dt_Omega}, it suffices to replace $x$ by $\phi_t(x)$ in \eqref{eq:dx_Omega} and \eqref{eq:dt_Omega_2}, and plug them into \eqref{eq:dt_Omega_1}.
\end{proof}

We now prove the following Proposition from which, given Proposition~\ref{bkm}, Theorem~\ref{mainthm} follows.

\begin{prop}
There exist a pair of smooth functions  $\rho_0$ and $\omega_0$  supported in $[\frac{1}{4},\frac{3}{4}]$, such that there is no global classical solution to \eqref{eq:system} with initial data $(\rho_0, \omega_0)$.
\label{prop:monotone_blowup}
\end{prop}

\begin{proof}
\textbf{Step 1.} We construct a pair of initial data $(\rho_0, \omega_0)$ as follows. Let  $\rho_0$ be smooth, nonnegative, supported in $[\frac{1}{4}, \frac{3}{4}]$, with $\max \rho_0 = \rho_0(\frac{1}{2}) = 2$, and $\rho_0(\frac{1}{3})=1$. Moreover, assume $\rho_0$ is increasing in $[\frac{1}{4}, \frac{1}{2}]$, and decreasing in $[\frac{1}{2}, \frac{3}{4}]$. Let $\omega_0$ be smooth, nonnegative, supported in $[\frac{1}{4}, \frac{1}{2}]$, with $\omega_0\equiv M$ in $[0.3, 0.45]$, where $M$ is a large constant to be determined later. Figure \ref{fig:init_2} gives a sketch of the initial data.

\begin{figure}[h!]
\begin{center}
\includegraphics[scale=0.8]{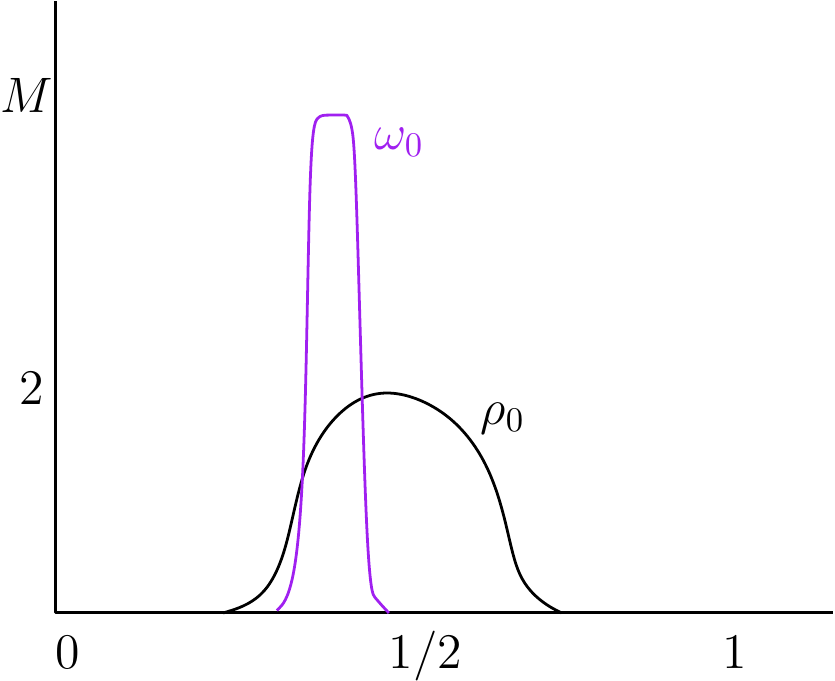}

\caption{\label{fig:init_2} A sketch of the initial data $(\rho_0, \omega_0)$.}
\end{center}
\end{figure}

Towards a contradiction we assume that there is a global classical solution. Let us first make a few observations.  Note that for all $x\in(0,1)$, the characteristic $\phi_t(x)$ must be well-defined for all time, and $\rho$ is conserved along $\phi_t(x)$, i.e. $\rho(t, \phi_t(x)) = \rho_0(x)$. Moreover, for all $t\geq 0$, we have
\begin{equation}\omega(x,t) \leq 0 \text{ for }x\in[\phi_t(1/2),1].
\label{eq:omega_temp}
\end{equation}
To see this, recall that by definition, $\rho_0$ is decreasing in $[\frac{1}{2}, 1]$. If there is a global classical solution, then the characteristics do not cross, hence for all $t\geq 0$, we have $\rho_x(x,t) \leq 0$ in $[\phi_t(1/2),1]$. We then obtain \eqref{eq:omega_temp} as a direct consequence, since the time derivative of $\omega$ along the characteristics $\phi_t(x)$ is equal to $\rho_x$.

Moreover, we have that $\phi_t(1/2)$ is increasing for all $t$. Note that
$$\frac{d}{dt} \phi_t(1/2) = -\phi_t(1/2) \Omega(t, \phi_t(1/2)) = -\phi_t(1/2)\int_{\phi_t(1/2)}^1 \frac{\omega(y,t)}{y}dy,$$
which is always non-negative due to \eqref{eq:omega_temp}.

\vspace{0.1cm}

\noindent\textbf{Step 2.} Our goal is to find a point $x_\infty$ (with $\rho_0(x_\infty)>0$) and a finite time $T$, such that $\phi_T(x_\infty) = 0$. This would imply that the classical solution has to break down at (or before) time $T$.
To show this, the main idea is to consider a family of characteristics originating from a sequence of points $\{x_n\}$. Let $x_1 = 1/3$ (recall that we let $\rho_0(\frac{1}{3}) =1)$. For $n>1$, find $x_n \in [0,\frac{1}{2}]$, such that $\rho_0(x_n) = \frac{1}{2} + 2^{-n}$. Observe that we have $x_1>x_2>x_3>\cdots$ since $\rho_0$ is increasing in $[0,\frac{1}{2}]$. Denote $x_\infty := \lim_{n\to\infty} x_n$, and it follows that $x_\infty>0$ and $\rho(x_\infty)=1/2$. The choice of $\{x_n\}$ is illustrated in Figure \ref{fig:init_1}.

Also, we choose $M$ large enough such that $C_0 := \Omega(0,x_1)=\int_{1/3}^1 \frac{\omega_0(y)}{y} dy>20$ (e.g. $M=200$ should work). Note that at $t=0$, $\Omega(0,x)$ is decreasing in $x$ due to the non-negativity of $\omega_0$. This implies that at $t=0$, we have $\Omega(0, x_n)> 20$ for all $n\geq 1$.

\begin{figure}[h!]
\begin{center}
\includegraphics[scale=1]{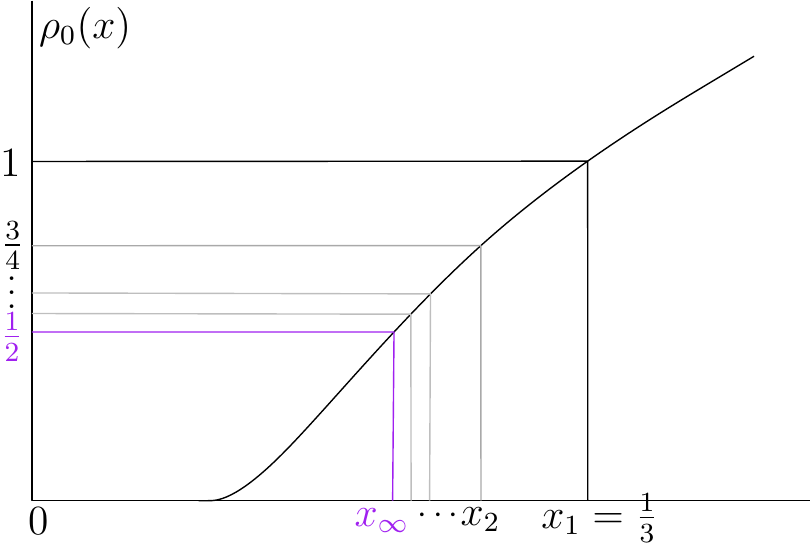}

\caption{\label{fig:init_1} A sketch of the choice of $\{x_n\}$.}
\end{center}
\end{figure}

Let us denote  $\rho_n := \rho_0(x_n)$, $\Phi_n(t):=\phi_t(x_n)$, $\Omega_n(t):=\Omega(t,\Phi_n(t))$. Observe that $\frac{d}{dt}\Phi_n(t)=u(t,\Phi_n(t))=-\Phi_n(t)
\Omega_n(t)$. Denoting $\psi_n(t):=-\ln\Phi_n(t)$ for $n\geq1$, we get
$$\frac{d}{dt}\psi_n(t)=\Omega_n(t).$$

 To see how $\Omega_n(t)$ grows in time, we apply Lemma \ref{lemma:Omega_t} to $x_n$, and use the fact that $\phi_t(1/2) \geq 1/2$ for all $t\geq 0$. This gives
\begin{equation}
\begin{split}
\frac{d}{dt} \Omega_n(t) &\geq \int_{\Phi_n(t)}^{\phi_t(1/2)} \frac{\partial_x \rho(t,y)}{y} dy +   \int_{\phi_t(1/2)}^{1} \frac{\partial_x \rho(t,y)}{y} dy\\
&\geq \int_{\Phi_n(t)}^{\phi_t(1/2)} \underbrace{\frac{\partial_x \rho(t,y)}{y}}_{\geq 0} dy + \underbrace{\frac{1}{\phi_t(1/2)}}_{\leq 2} \underbrace{\big(\rho(t,1)-\rho(t,\phi_t(1/2)\big)}_{=-2}\\
&\geq \int_{\Phi_n(t)}^{\phi_t(1/2)} \frac{\partial_x \rho(t,y)}{y} dy - 4 \quad \text{ for all }n\geq 1.
\end{split}
\label{eq:dt_Omega_crude}
\end{equation}
Recall that $\omega_0$ is chosen such that $\Omega_n(0)\geq 20$, hence \eqref{eq:dt_Omega_crude} immediately implies
$\Omega_n(t) \geq 0 $ for all $n$ and all $t\in[0,5)$. Since $\frac{d}{dt} \psi_n(t) = \Omega_n(t)$, we have that $\psi_n(t)$ is increasing for $t\in[0,5)$.

For $n\geq 2$, using \eqref{eq:dt_Omega_crude}, we have
\begin{equation}
\begin{split}
\frac{d}{dt} \Omega_n(t) &\geq \int_{\Phi_n(t)}^{\Phi_{n-1}(t)} \frac{\partial_x \rho(t,y)}{y} dy - 4\\
&\geq (\rho_{n-1}-\rho_n)\frac{1}{\Phi_{n-1}(t)} - 4\\
& = 2^{-n} e^{\psi_{n-1}(t)} - 4.
\end{split}
\label{eq:dt_Omega_3}
\end{equation}
Collecting everything together, we arrive at the following system of inequalities:
         \begin{equation}
         \begin{cases}& \psi^{\prime\prime}_n(t)\geq
         2^{-n} e^{\psi_{n-1}(t)}-4\\[0.2cm]
         & \psi^{\prime}_{n-1}(t) = \Omega_{n-1}(t) \geq 0\\[0.2cm]
         & \psi_n(t)\geq \psi_{n-1}(t) \geq 0
         \end{cases} \quad\text{ for $n\geq 2$, $0\leq t< 5$.}
         \label{eq:collection_2}
         \end{equation}
       \textbf{Step 3.}  Take $t_1=1$, and let $t_{n+1}=t_{n}+2^{-n}$ and $\tilde{t}_n=t_n+2^{-(n+1)}$ for $n\geq 1$. Let $T:=\lim_{n\rightarrow \infty} t_n=2$ (Note that \eqref{eq:collection_2} holds until $t=5$, hence it holds for all $t\leq T$). We will show that $a_n:= \psi_n(t_n) \to \infty$ as $n\to \infty$.

  Take $n\geq 2$. Since $\psi_{n-1}(t)$ is increasing in $t$ for all $t< 5$, we have
       \begin{equation}
       \psi^{\prime\prime}_n(t)\geq
        2^{-n} e^{\psi_{n-1}(t_{n-1})}-4\quad \text{ for }t_{n-1}\leq t < 5.
        \label{eq:temp_psi}
        \end{equation}
   This implies that for $\tilde{t}_{n-1}\leq t \leq t_n$ (note that all $t_n$'s are less than 5),
        \begin{equation*}
        \begin{split}
        \psi^\prime_n(t)& \geq \psi^\prime_n(\tilde{t}_{n-1})-4(t_n-\tilde t_{n-1})\quad \text{(since $\psi_n''(t)\geq -4$)}\\
        &\geq (2^{-n} e^{\psi_{n-1}(t_{n-1})}-4)
        (\tilde{t}_{n-1}-t_{n-1})+\psi^\prime_n(t_{n-1})-4(t_n-\tilde t_{n-1}) \quad\text{(using \eqref{eq:temp_psi})}\\
        &\geq ( 2^{-n} e^{\psi_{n-1}(t_{n-1})}-4) 2^{-n} - 4\cdot 2^{-n}\\
        &= ( 2^{-n} e^{\psi_{n-1}(t_{n-1})}-8) 2^{-n}.
        \end{split}
        \end{equation*}
Once we have the lower bound for $\psi_n'(t)$ for $\tilde t_{n-1} \leq t \leq t_n$, we can use it get a lower bound for $\psi_n(t_n)$ as follows:
\begin{equation}
\begin{split}
\psi_n(t_n)&\geq  ( 2^{-n} e^{\psi_{n-1}(t_{n-1})}-8) 2^{-n}\cdot(t_n-\tilde{t}_{n-1})+\psi_n(\tilde{t}_{n-1})\\
&\geq ( 2^{-n} e^{\psi_{n-1}(t_{n-1})}-8) 2^{-2n} +\psi_{n-1}(\tilde t_{n-1})\\
&\geq ( 2^{-n} e^{\psi_{n-1}(t_{n-1})}-8) 2^{-2n}
        +\psi_{n-1}(t_{n-1}),
\end{split}
\label{eq:psi_temp}
\end{equation}
where in the second inequality we used the fact that $\psi_n \geq \psi_{n-1}$, and in the last inequality we used that $\phi_{n-1}$ is increasing for $t\leq 5$, hence $\phi_{n-1}(\tilde t_{n-1}) \geq \phi_{n-1}( t_{n-1})$.

\noindent \textbf{Step 4}. Denoting $a_n:= \psi_n(t_n)$, we obtain the following recursive relation from \eqref{eq:psi_temp}:
\begin{equation*}
\begin{split}
a_n &\geq 2^{-2n}(2^{-n}e^{a_{n-1}}-8)+a_{n-1}\\
&\geq e^{a_{n-1}-3n}-1+a_{n-1}\quad \text{ for $n\geq 2$}.
\end{split}
\end{equation*}
One can then use induction to show that if $a_1 \geq 9$, then $a_n\geq 3n+6$ for all $n\geq 1$, hence $a_n\to\infty$ as $n\to\infty$.

Finally it remains to check that whether $a_1\geq 9$ is satisfied, i.e. whether $\psi_1(1) \geq 9$. Recall that $\psi_1(0)\geq 0$, and $\psi_1'(t) = \Omega_1(t)$, with $\Omega_1(t)\geq 20$ and $\Omega_1'(t)\geq -4$. Hence we have $\Omega_1(t)\geq 16$ for $0\leq t\leq 1$, which gives $\psi_1(1)\geq 16$, and this concludes the proof.
\end{proof}

{\bf Acknowledgement.} KC has been partially supported by the National Science Foundation (NSF) grant DMS-1159133. AK and YY have been partially supported by the NSF-DMS grants 1104415 and 1159133.
AK acknowledges support of the Guggenheim Fellowship and thanks Guo Luo and Vladimir Sverak for useful discussions.


\begin{thebibliography}{99}

\bibitem{BKM}
J. T. Beale, T. Kato, and A. Majda,  \it Remarks on the breakdown of smooth solutions for the 3-D Euler equations, \rm
Commun. Math. Phys., {\bf 94}, pp. 61--66, 1984

\bibitem{CaoWu1} C.~Cao and J.~Wu, \it Global regularity for the 2D anisotropic Boussinesq equations with vertical dissipation,
\rm Archive for Rational Mechanics and Analysis, {\bf 208} (2013), 985--1004

\bibitem{Chae} D.~Chae, \it Global regularity for the 2D Boussinesq equations with partial viscosity
terms, \rm Adv. Math. {\bf 203} (2006), 497--513

\bibitem{ChaeNam} D.~Chae and H. Nam, \it Local existence and blow-up criterion for the {B}oussinesq
              equations, \rm Proc. Roy. Soc. Edinburgh Sect. A.
              {\bf 127} (1997), 935--946
    \bibitem{CF1} P.~Constantin and C.~Fefferman, \it Direction of vorticity and the problem of
global regularity for the Navier-Stokes equations, \rm Indiana Univ. Math. J. {\bf 42} (1993), 775--789

\bibitem{CFM}
P.~Constantin, C.~Fefferman and A.~Majda, \it
Geometric constraints on potentially singular solutions for the
$3$-D Euler equation, \rm
Communications in PDE, {\bf 21} (1996), pp. 559--571

\bibitem{EShu} W. E and C. Shu, \it Samll-scale structures in Boussinesq convection, \rm Phys. Fluids {\bf 6}
(1994), 49--58

\bibitem{Gibbon} J.D.~Gibbon, \it The three-dimensional Euler equations: Where do we stand? \rm
Physica D {\bf 237} (2008), 1894--1904

\bibitem{HouLi} T.~Hou and C.~Li, \it Global well-posedness of the viscous Boussinesq equations, \rm  Discrete
Contin. Dyn. Syst. {\bf 12} (2005), 1--12

\bibitem{HL1}
T. ~Y. Hou and R. Li,
\it Dynamic depletion of vortex stretching and non-blowup of the
3-D incompressible Euler equations,
\rm {J. Nonlinear Science}, \textbf{16} (2006), 639--664

\bibitem{HL2}
T. Y. Hou and R. Li, \it Blowup or no blowup? The interplay between theory and numerics,
\rm{Physica D}, \textbf{237}, pp. 1937--1944, 2008

\bibitem{HouLuo1} T.~Hou and G.~Luo, \it Potentially Singular Solutions of the 3D Incompressible Euler Equations,  \rm preprint arXiv:1310.0497


\bibitem{HouLuo2} T.~Hou and G.~Luo, \it On the finite-time blow up of a 1D model for the 3D incompressible Euler equations,  \rm preprint arXiv:1311.2613

\bibitem{KS} A.~Kiselev and V.~Sverak, \it Small scale creation for solutions of the incompressible two dimensional Euler equation, \rm preprint arXiv:1310.4799

\bibitem{MB} A.~Majda and A.~Bertozzi, \it Vorticity and
Incompressible Flow, \normalfont Cambridge University Press, 2002

\bibitem{PumSig} A.~Pumir and E.D. Siggia, \it Development of singular solutions to the axisymmetric
Euler equations, \rm Phys. Fluids A {\bf 4} (1992), 1472--1491

\bibitem{Yud2000} V.I.~Yudovich, \it Eleven great problems of mathematical hydrodynamics, \rm Moscow Mathematical Journal,
{\bf 3}(2003), 711--737

\end{thebibliography}
\end{document}